\documentclass{amsart}

\usepackage{graphicx}
\usepackage{epstopdf}
\newtheorem{theorem}{Theorem}[section]
\newtheorem{lemma}[theorem]{Lemma}
\newtheorem{corollary}[theorem]{Corollary}

\theoremstyle{definition}

\theoremstyle{remark}

\numberwithin{equation}{section}

\usepackage{amssymb} % Required for circledast

\begin{document}

\title{Karhunen-Lo\`eve decomposition of Gaussian measures on Banach spaces}


%    author one information
\author[X. Bay]{Xavier Bay}
\address{Mines Saint-Etienne, 158 cours Fauriel, 42023 Saint-Etienne, France (LIMOS UMR 6158)}
%\curraddr{}
\email{xavier.bay@emse.fr}
\thanks{}

%    author two information
\author[J.-C. Croix]{Jean-Charles Croix}
\address{Mines Saint-Etienne, 158 cours Fauriel, 42023 Saint-Etienne, France (LIMOS UMR 6158)}
%\curraddr{}
\email{jean-charles.croix@emse.fr}
\thanks{Part of this research was conducted within the frame of the Chair in Applied Mathematics OQUAIDO, gathering partners in technological research (BRGM, CEA, IFPEN, IRSN, Safran, Storengy) and academia (CNRS, Ecole Centrale de Lyon, Mines Saint-Etienne, University of Grenoble, University of Nice, University of Toulouse) around advanced methods for Computer Experiments.}

%    \subjclass is required.

\date{}

\dedicatory{}

%    "Communicated by" -- provide editor's name; required.
%\commby{Zhen-Qing Chen}

%    Abstract is required.
\begin{abstract}
	The study of Gaussian measures on Banach spaces is of active interest both in pure and applied mathematics. In particular, the spectral theorem for self-adjoint compact operators on Hilbert spaces provides a canonical decomposition of Gaussian measures on Hilbert spaces, the so-called Karhunen-Lo\`eve expansion. In this paper, we extend this result to Gaussian measures on Banach spaces in a very similar and constructive manner. In some sense, this can also be seen as a generalization of the spectral theorem for covariance operators associated to Gaussian measures on Banach spaces. In the special case of the standard Wiener measure, this decomposition matches with Paul L\'evy's construction of Brownian motion.
\end{abstract}

\maketitle

\section{Preliminaries on Gaussian measures}

Let us first remind a few properties of Gaussian measures on Banach spaces. Our terminology and notations are essentially taken from \cite{Bog1998} (alternative presentations can be found in \cite{Kuo1975}, \cite{Vak1987} or \cite{Hai2009}). In this work, we consider a separable Banach space $X$, equipped with its Borel $\sigma$-algebra $\mathcal{B}(X)$. Note that every probability measure on $(X,\mathcal{B}(X))$ is Radon and that Borel and cylindrical $\sigma$-algebras are equal in this setting.
\newline
A probability measure $\gamma$ on $(X,\mathcal{B}(X))$ is Gaussian if and only if for all $f\in X^*$ (the topological dual space of $X$), the pushforward measure $\gamma\circ f^{-1}$ (of $\gamma$ through $f$) is a Gaussian measure on $(\mathbb{R},\mathcal{B}(\mathbb{R}))$. Here, we only consider the case $\gamma$ centered for simplicity (the general case being obtained through a translation). An important tool in the study of a (Gaussian) measure is its characteristic functional $\hat{\gamma}$ (or Fourier transform)
\begin{equation*}
\hat{\gamma}:f\in X^*\to\hat{\gamma}(f)=\int_Xe^{i\langle x,f\rangle_{X,X^*}}\gamma(dx)\in\mathbb{C},
\end{equation*}
where $\langle.,.\rangle_{X,X^*}$ is the duality pairing. Since $\gamma$ is a centered Gaussian measure, we have
\begin{equation}
\forall f\in X^*,\;\hat{\gamma}(f)=\exp\left(-\frac{C_\gamma(f,f)}{2}\right)
\label{eq_characteristic_functional},
\end{equation}
where $C_\gamma$ is the covariance function
\begin{equation*}
C_\gamma:(f,g)\in X^*\times X^*\to\int_X\langle x,f\rangle_{X,X^*}\langle x,g\rangle_{X,X^*}\gamma(dx)\in\mathbb{R}.
\end{equation*}
One of the most striking results concerns integrability. Indeed, using a rotation invariance principle, it has been shown that a Gaussian measure $\gamma$ admits moments (in a Bochner sense) of all orders (as a simple corollary of Fernique's theorem, see \cite{Bog1998}). Consequently, its covariance operator may be defined as
\begin{equation*}
R_\gamma: f\in X^*\to\int_X\langle x,f\rangle_{X,X^*} x\gamma(dx)\in X,
\end{equation*}
using Bochner's integral and is characterized by the following relation
\begin{equation}
\forall (f,g)\in X^*\times X^*,\;\langle R_\gamma f,g\rangle_{X,X^*}=C_\gamma(f,g).
\label{eq_covariance_operator}
\end{equation}
Most noticeably, $R_\gamma$ is a symmetric positive kernel (Hilbertian or Schwartz kernel) in the following sens:
\begin{eqnarray*}
	\forall (f,g)\in X^*\times X^*,\;\langle R_\gamma f,g\rangle_{X,X^*}&=&\langle R_\gamma g,f\rangle_{X,X^*},\\
	\forall f\in X^*,\;\langle R_\gamma f,f\rangle_{X,X^*}&\geq& 0.
\end{eqnarray*}
Furthermore, the Cameron-Martin space $H(\gamma)$ associated to $\gamma$ is the Hilbertian subspace of $X$ with Hilbertian kernel $R_\gamma$ (see \cite{Sch1964} and \cite{Aro1950} for the usual case of reproducing kernel Hilbert spaces). 
In particular, we will extensively use the so-called \textit{reproducing property}
\begin{equation*}
\forall h\in H(\gamma),\;\forall f\in X^*,\;\langle h,f\rangle_{X,X^*}=\langle h,R_\gamma f\rangle_\gamma,
\end{equation*}
where $<.,.>_\gamma$ denotes the inner product of $H(\gamma)$. Note that $H(\gamma)$ is continuously embedded in $X$ and admits $R_\gamma(X^*)$ as a dense subset. Additionally, the covariance operator has been shown to be nuclear and in particular compact (see \cite{Vak1987}, Chapter 3 for a detailed presentation and proofs).
\newline
Our objective is to decompose any Gaussian measure $\gamma$ on a (separable) Banach space $X$ which, in fact, can be done by considering any Hilbert basis of the Cameron-Martin space $H(\gamma)$. Indeed, let $(h_n)_n$ be any arbitrary orthonormal basis of $H(\gamma)$ and $(\xi_n)_n$ a sequence of independent standard Gaussian random variables defined on a probability space $(\Omega,\mathcal{F},\mathbb{P})$. Then the series
\begin{equation}
\sum_{n}\xi_n(\omega)h_n
\label{random_series},
\end{equation}
converges almost surely in $X$ and the distribution of its sum is the Gaussian measure $\gamma$ (cf. theorem 3.5.1 p. 112 in \cite{Bog1998}). When $X$ is a Hilbert space, a canonical Hilbert basis of the Cameron-Martin space $H(\gamma)$ is given by the spectral decomposition of the covariance operator $R_\gamma$ as a self-adjoint compact operator on $X$ (see Mercer's theorem in the special case $X=L^2[a,b]$ with $[a,b]$ any compact interval of $\mathbb{R}$). In this paper, we will show how to define and construct such a basis in the general case $X$ Banach by a direct generalization of the Hilbert case. In particular, this "diagonalizing" basis will be of the form $h_n=R_\gamma h_n^*$ where $h_n^*$ is in the dual space $E^*$ for all $n$. As a special case of the representation (\ref{random_series}), the corresponding decomposition in $X$ (for the strong topology) will be
\begin{equation*}
x=\sum_{n}\langle x,h_n^*\rangle_{X,X^*}h_n,
\end{equation*}
$\gamma$ almost everywhere (since $(h_n^*)_n$ is a sequence of independent standard normal random variables by the reproducing property). Roughly speaking, it means that $\gamma$ can be seen as the countable product of the standard normal distribution $N(0,1)$ on the real line:
\begin{equation*}
\gamma=\bigotimes_n\\N(0,1).
\end{equation*}
For a recent review of the interplay between covariance operators and Gaussian measures decomposition, consult \cite{Kva2014}. To see how to construct such a basis, we start with the Hilbert case.

\section{Gaussian measures on Hilbert spaces}
\label{sec_hilbert}

Hilbert geometry has nice features that are well understood, including Gaussian measures structure (see \cite{Kuo1975} and \cite{Dap2006} for a recent treatment). First of all, Riesz representation theorem allows to identify $X^*$ with $X$. As a linear operator on a Hilbert space, the covariance operator $R_\gamma$ of a Gaussian measure $\gamma$ is self-adjoint and compact. Spectral theory exhibits a particular Hilbert basis of $X$ given by the set $(x_n)_n$ of eigenvectors of $R_\gamma$. Using this specific basis, the covariance operator is
\begin{equation*}
R_\gamma: x\in X\to\sum_{n}\lambda_n\langle x,x_n\rangle_X x_n\in X,
\end{equation*}
where $\langle.,.\rangle_X$ is the inner product of $X$. A simple normalization, namely $h_n=\sqrt{\lambda_n}x_n$, provides a Hilbert basis of $H(\gamma)$. The nuclear property of $R_\gamma$ simplifies to
\begin{equation*}
\sum_{n}{\lVert h_n\rVert^2_X}=\sum_{n}\lambda_n<+\infty.
\end{equation*}
Using the terminology of random elements, let $Y$ be the infinite-dimensional vector defined almost surely by
\begin{equation*}
	Y(\omega)=\sum_{n}\xi_n(\omega)h_n=\sum_{n}\sqrt{\lambda_{n}}\xi_n(\omega)x_n,
\end{equation*}
where $(\xi_n)_n$ is a sequence of independent standard normal random variables. Then $\gamma$ is the distribution of the Gaussian vector $Y$. In the context of stochastic processes, this representation is well-known as the Karhunen-Lo\`eve expansion (\cite{Kar1947}, \cite{Loe1960}) of the process $Y=(Y_t)_{t \in T}$ (assumed to be mean-square continuous over a compact interval $T=[a,b]$ of $\mathbb{R}$).
\newline
In order to extend this spectral decomposition to the Banach case, let us recall the following simple property (where $B_X$ denotes the unit closed ball of $X$):
\begin{equation}
\lambda_0=\sup_{x\in X\setminus\{0\}}\frac{\langle R_\gamma x,x\rangle_X}{\lVert x\rVert_X^2}=\max_{x\in B_{X}}\langle R_\gamma x,x\rangle_X
\end{equation}
is the largest eigenvalue of $R_\gamma$ and is equal to the Rayleigh quotient $\frac{\langle R_\gamma x_0,x_0\rangle_X}{\lVert x_0\rVert_X^2}$ where $x_0$ is any corresponding eigenvector. A similar interpretation is valid for every $n\in\mathbb{N}$:
\begin{equation*}
 \lambda_n=\max_{x\in B_X\cap span(x_0,...,x_{n-1})^\perp}\langle R_\gamma x,x\rangle_X.
\end{equation*}
Keeping this interpretation in mind, we can now consider the Banach case.

\section{Gaussian measures in Banach spaces}

In the context of Banach spaces, the previous spectral decomposition of the covariance operator doesn't make sense anymore. Nevertheless, we will show in section \ref{sec_principle} that the Rayleigh quotient is well defined in this context (lemma \ref{lem_existence}). Combining this and a simple decomposition method (lemma \ref{lem_split}), we give in section \ref{sec_construction} an iterative decomposition scheme of a Gaussian measure. Main analysis and results are given in the last section \ref{sec_analysis}.

\subsection{Rayleigh quotient and split decomposition}
\label{sec_principle}

The first lemma in this section is an existence result of particular linear functionals based on a compactness property. The second one provides a method to separate a Banach space into two components with respect to a linear functional and a Gaussian measure. These results are given independently to emphasize that lemma \ref{lem_split} could be combined with different linear functionals to define other iterative decomposition schemes (see section \ref{sec_construction}).    

\begin{lemma}
	Let $\gamma$ be a Gaussian measure on $(X,\mathcal{B}(X))$ a separable Banach space and set $\lambda_0=\sup_{f\in B_{X^*}}\langle R_\gamma f,f\rangle_{X,X^*}\in[0,+\infty]$. Then
	\begin{equation*}
	\exists f_0\in B_{X^*},\;\lambda_0=\langle R_\gamma f_0,f_0\rangle_{X,X^*}.
	\end{equation*}
	Moreover, we may assume $\lVert f_0\rVert_{X^*}=1$.
	\label{lem_existence}
\end{lemma}

\begin{proof}[Proof of lemma \ref{lem_existence}]
	Let $(f_n)_n\in B_{X^*}$ be a maximizing sequence:
	\begin{equation*}
	\langle R_\gamma f_n,f_n\rangle_{X,X^*}\to\lambda_0\in[0,+\infty].
	\end{equation*}
	From the weak-star compactness of $B_{X^*}$ (see Banach-Alaoglu theorem), we can suppose that $f_n\rightharpoonup f_\infty$ for the $\sigma(X^*,X)$-topology where $f_\infty \in B_{X^*}$ . This implies that
	\begin{equation*}
	\hat{\gamma}(f_n)=\int_{X}e^{i\langle x,f_n\rangle_{X,X^*}}\gamma(dx)\to\int_{X}e^{i\langle x,f_\infty\rangle_{X,X^*}}\gamma(dx)=\hat{\gamma}(f_\infty),
	\end{equation*}
	using Lebesgue's convergence theorem. From equations \ref{eq_characteristic_functional} and \ref{eq_covariance_operator}, we conclude that $\langle R_\gamma f_n,f_n\rangle_{X,X^*}\to\langle R_\gamma f_\infty,f_\infty\rangle_{X,X^*}$. Hence $\lambda_0=\langle R_\gamma f_\infty,f_\infty\rangle_{X,X^*}\in\mathbb{R}_{+}$. If $\lambda_0>0$, then $\lVert f_\infty\rVert_{X^*}=1$ and we can take $f_0=f_\infty$. In the degenerate case $\lambda_0=0$, we have $R_\gamma =0$ and any $f_0$ of unit norm is appropriate.
\end{proof}

We will now show how to split both $X$ and $\gamma$, given any $f\in X^*$ of non trivial Rayleigh quotient (in the previous sense).

\begin{lemma}
	Let $\gamma \neq \delta_0$ be a non trivial Gaussian measure on a separable Banach space $(X,\mathcal{B}(X))$. Pick $f_0\in X^*$ such that $\lVert f_0\rVert_{X^*}=1$ and $\lambda_0=\langle R_\gamma f_0,f_0\rangle_{X,X^*}>0$. Set $P_0:x\in X\to \langle x,f_0\rangle_{X,X^*}x_0$, $R_\gamma f_0=\lambda_0x_0$ and $h_0=\sqrt{\lambda_0}x_0$, then we have the following properties.
	\begin{enumerate}
		\item $\langle x_0,f_0\rangle_{X,X^*} = 1$ and $\lVert h_0\rVert_{\gamma}=1.$
		\item $P_0$ is the projection on $X$ with range $\mathbb{R}x_0$ and null space $\ker(f_0)=\lbrace x\in X,\;\langle x,f_0\rangle_{X,X^*}=0\rbrace$. Furthermore, the restriction $Q_0$ of $P_0$ on $H(\gamma)$ is the orthogonal projection onto $\mathbb{R}h_0$:
		\begin{equation*}
		h\in H(\gamma),\;\langle h,f_0\rangle_{X,X^*}x_0=\langle h,h_0\rangle_\gamma h_0.
		\end{equation*}
		\item According to the decomposition $x=P_0x+(I-P_0)x$ in $X$, the Gaussian measure $\gamma$ can be decomposed as
		\begin{equation*}
		\gamma = \gamma_{\lambda_0}*\gamma_1,
		\end{equation*}
		where $\gamma_{\lambda_0}=\gamma\circ P_0^{-1}$ and $\gamma_1=\gamma_0\circ (I-P_0)^{-1}$ are Gaussian measures with respective covariance operators:
		\begin{eqnarray*}
			R_{\lambda_0}:f\in X^*&\to&\lambda_0\langle x_0,f\rangle_{X,X^*}x_0,\\
			R_{\gamma_1}:f\in X^*&\to&R_{\gamma}f-R_{\lambda_0}f.
		\end{eqnarray*}
		In particular,
		\begin{eqnarray*}
		 R_{\gamma}f=\lambda_0\langle x_0,f\rangle_{X,X^*}x_0+R_{\gamma_1}f.
		 \end{eqnarray*}
		\item The Cameron-Martin space $H(\gamma)$ is decomposed as
		\begin{equation*}
		H(\gamma)=\mathbb{R}h_0\oplus H(\gamma_1),
		\end{equation*}
		where $H(\gamma_1)=(I-Q_0)(H(\gamma))=(\mathbb{R}h_0)^\perp$ equipped with the inner product of $H(\gamma)$ is the Cameron-Martin space of $\gamma_1$.
		\item For each $t\in\mathbb{R}$, denote by $tx_0+\gamma_1$ the Gaussian measure on $X$ centered at $tx_0$ with covariance operator $R_{\gamma_1}$. Then, $\gamma^t$ is the conditional probability distribution of $x\in X$ given $f_0(x)=t$:
		\begin{equation*}
		\forall B\in\mathcal{B}(X),\;\gamma^t(B)=\gamma_1(B-tx_0)=\gamma(B\vert f_0=t).
		\end{equation*}
		Moreover, $f_0$ is $\mathcal{N}(0,\lambda_0)$ and the deconditioning formula is as follows:
		\begin{equation*}
		 \gamma(B)=\int_{\mathbb{R}}\gamma^t(B)\frac{e^{-\frac{t^2}{2\lambda_0}}}{\sqrt{2\pi\lambda_0}}dt.
		\end{equation*}
			\end{enumerate}
	\label{lem_split}
\end{lemma}

The proof is straightforward and is given in the appendix. Concerning the last property on conditioning, it is worth noting that the conditional covariance operator $R_{\gamma_1}$ does not depend of the particular value $t$ of the random variable $f_0\in X^*$. We will now use both of these lemmas to build a complete decomposition of any Gaussian measure $\gamma$.

\subsection{Iterative decomposition of a Gaussian measure}
\label{sec_construction}

Consider a (centered) Gaussian measure $\gamma$ on a separable Banach space $(X,\mathcal{B}(X))$. The initial step of the decomposition is to split $X$ and $\gamma$ according to lemma \ref{lem_split} using $f_0\in X^*$ given by lemma \ref{lem_existence}. The same process is applied to the \textit{residual} Gaussian measure $\gamma_1$ defined in lemma \ref{lem_split}, and so on and so forth. Now, we formalize the resulting iterative decomposition scheme. \\

Define $\gamma_0 = \gamma$ (initialization). By induction on $n \in \mathbb{N}$ (iteration), we define the Gaussian measure $\gamma_{n+1}$ of covariance operator $R_{\gamma_{n+1}}$ such that
\begin{equation*}
\forall f\in X^*,\;R_\gamma f = \sum_{k=0}^{n}\lambda_k\langle x_k,f\rangle_{X,X^*}x_k +R_{\gamma_{n+1}}f
\end{equation*}
where $\lambda_n =\max_{f\in B_{X^*}}\langle R_{\gamma_n}f,f\rangle_{X,X^*}$ and where $x_n$ is defined by the relation $R_{\gamma_n}f_n = \lambda_n x_n$ with $f_n$ chosen such that $\lambda_n = \langle R_{\gamma_n}f_n,f_n\rangle_{X,X^*}$.\\
    
From lemma \ref{lem_split}, we have the orthogonal decomposition for all $n$
\begin{equation*}
H(\gamma)=span(h_0, ...,h_n)\oplus H(\gamma_{n+1})
\end{equation*}
where 	$h_n =\sqrt{\lambda_n}x_n.$
If for some $n$, $\lambda_{n+1}=0$, then $R_{\gamma_{n+1}} = 0$ and $H(\gamma_{n+1})=\lbrace 0\rbrace$, which means that $R_{\gamma}$ is a finite-rank operator and $H(\gamma)=span(h_0, ...,h_{n}) = span(x_0, ...,x_{n})$ a finite-dimensional linear space. This means that $\gamma$ is a finite-dimensional Gaussian measure with support equal to its Cameron-Martin space. Theorem \ref{thm_decomposition} gives the properties of this decomposition in the general case where $H(\gamma)$ is infinite-dimensional.

\begin{theorem}
	Suppose $H(\gamma)$ is infinite-dimensional and keep previous notations, we have the following properties.
	\begin{enumerate}
		\item $(h_n)_n$ is an orthonormal sequence in $H(\gamma)$.
		\item $(x_n)_n$ and $(f_n)_n$ are satisfying the following relations:
		\begin{enumerate}
			\item $\forall n\in\mathbb{N},\;\lVert x_n\rVert_{X}=\langle x_n,f_n\rangle_{X,X^*} = 1$,
			\item $\forall(k,l)\in\mathbb{N}^2,\;k>l,\;\langle x_k,f_l\rangle_{X,X^*} = 0$.
		\end{enumerate}
		\item Let $Q_n : h\in H(\gamma)\to Q_{n}h=\sum_{k=0}^{n}\langle h,h_k\rangle_{\gamma}h_k$ be the orthogonal projection onto the linear space $span(h_0, ..., h_n) = span(x_0, ..., x_n)$ in $H(\gamma)$. Then, we have $Q_{n}h=\sum_{k=0}^{n}\langle h - Q_{k-1}h,f_k\rangle_{X,X^*}x_k,$ with the convention that $Q_{-1} = 0$.
		\item Define $P_{n}$ on $X$ by $ P_{n}x=\sum_{k=0}^{n}\langle x - P_{k-1}x,f_k\rangle_{X,X^*}x_k,$ with the same convention $P_{-1} = 0$. Then, $P_n$ is the projection onto $span(x_0, ..., x_n)$ and null space $\lbrace x\in X:\;\langle x,f_k\rangle_{X,X^*}=0 \text{ for } k = 0, ..., n\rbrace.$ Furthermore, the operator $P_n$ restricted to $H(\gamma)$ is equal to $Q_n$.
		\item According to the decomposition $x=P_nx+(I-P_n)x$ in $X$, the Gaussian measure $\gamma$ can be decomposed as $\gamma = \gamma_{\lambda_0, ...,\lambda_n}*\gamma_{n+1}$		where $\gamma_{\lambda_0, ...,\lambda_n} =\gamma\circ P_n^{-1}$ is a Gaussian measure with covariance operator
		\begin{eqnarray*}
			R_{\lambda_0, ...,\lambda_n}:f\in X^* &\to& \sum_{k=0}^{n}\lambda_k\langle x_k,f\rangle_{X,X^*}x_k.
			\end{eqnarray*} 
			Furthermore, we have $\gamma_{n+1}=\gamma\circ (I-P_n)^{-1}$ and the relation
			\begin{equation*}
			 R_\gamma = R_{\lambda_0, ...,\lambda_n} + R_{\gamma_{n+1}}.
			\end{equation*}   
		\item The Cameron-Martin space $H(\gamma)$ is decomposed as
		\begin{equation*}
		H(\gamma)=span(h_0, ...,h_n)\oplus H(\gamma_{n+1}),
		\end{equation*} where $H(\gamma_{n+1})=(I-Q_n)(H(\gamma))$ equipped with the inner product of $H(\gamma)$ is the Cameron-Martin space of the Gaussian measure $\gamma_{n+1}$.
		\item Let $x_{n}^{*} = (I - P_{n-1})^{*}f_n$ for $n \geq 0$. Then, $\forall n, R_{\gamma}x_{n}^{*} = \lambda_n x_n$. The random variables $x_{n}^{*}$ are independent $\mathcal{N}(0,\lambda_n)$, and 
		\begin{equation*}
		\forall n,\;P_{n}x=\sum_{k=0}^{n}\langle x,x_{k}^{*} \rangle_{X,X^*}x_k.
		\end{equation*}
		For the computation of the dual basis $(x_{n}^{*})_n$, we have the recurrence formula
				\begin{equation*}  
				x_{n}^{*} = f_n - P_{n-1}^{*}f_n
				\end{equation*}
with $P_{n-1}^{*}f_n=\sum_{k=0}^{n-1}\langle x_{k},f_n \rangle_{X,X^*}x_{k}^{*}$ and  $x_{0}^{*}=f_0.$ \\
Furthermore, $\gamma_{\lambda_0, ...,\lambda_n} = \gamma_{\lambda_0}*...*\gamma_{\lambda_n}$ where $\gamma_{\lambda_n}$ is the distribution of the random vector $x \rightarrow \langle x,x_n^* \rangle_{X,X^*}x_n$ for all $n$.

		\item Let $h_{n}^{*} = \sqrt{\lambda_n}^{-1} x_{n}^{*}$ for $n \geq 0$. Then, we have $R_{\gamma}h_{n}^{*} = h_n,$ and the random variables $h_{n}^{*}$ are independent $\mathcal{N}(0,1)$.
		\item For each $t=(t_0,...,t_n)\in\mathbb{R}^{n+1}$, denote by $\sum_{k=0}^{n}t_kx_k+\gamma_{n+1}$ the Gaussian measure on $X$ centered at $\sum_{k=0}^{n}t_kx_k$ with covariance operator $R_{\gamma_{n+1}}$. Then, $\gamma^t=\sum_{k=0}^{n}t_kx_k+\gamma_{n+1}$ is the conditional probability distribution of $x\in X$ given $x_{0}^{*}(x)=t_0,...,x_{n}^{*}(x)=t_n$:
		\begin{equation*}
		\forall B\in\mathcal{B}(X),\;\gamma^t(B)=\gamma_{n+1}\left(B-\sum_{k=0}^{n}t_kx_k\right)=\gamma(B\vert x_{0}^{*}=t_0,...,x_{n}^{*}=t_n).
		\end{equation*}
		The deconditioning formula is
		\begin{equation*}
		 \gamma(B)=\int_{\mathbb{R}^n}\gamma^t(B)
		\prod_{k=0}^n
		\frac{e^{-\frac{t_{k}^2}{2\lambda_k}}}{\sqrt{2\pi\lambda_k}}dt_k
		\end{equation*}
		\label{assertion_conditionnement}
	\end{enumerate}
	\label{thm_decomposition}
\end{theorem}

This theorem is a straightforward extension of lemma \ref{lem_split} and a proof is given in the appendix. It remains to see that this decomposition is complete, namely that we have
\begin{equation*}
	\gamma=*_n\gamma_{\lambda_n}
\end{equation*}
according to the decomposition of the covariance operator
\begin{equation*}
	R_{\gamma}=\sum_{n}\lambda_n\langle x_n,.\rangle_{X,X^*}x_n.
\end{equation*}

\subsection{Asymptotic analysis.}
\label{sec_analysis}

In this section, we suppose that $H(\gamma)$ is infinite-dimensional and we use notations of the previous section. The two following lemmas will be essential for the main result of this paper (theorem \ref{thm_HilbertBasis}). 

\begin{lemma}
	We have $B_{H(\gamma_n)} = B_{H(\gamma)}\cap span(h_0,...,h_{n-1})^\perp$ for all $n$ and
	\begin{equation}
	\sqrt{\lambda_n}=\sup_{f\in B_{X^*}}\sup_{h\in B_{H(\gamma_n)}}\langle h,f\rangle_{X,X^*}.
	\label{equation_doublesup}
	\end{equation}
	\label{lem_doublesup}
\end{lemma}

\begin{proof}[Proof of lemma \ref{lem_doublesup}]
	Since $H(\gamma)=span(h_0,...,h_{n-1})\oplus H(\gamma_n)$ and $\lVert .\rVert_{\gamma_n}=\lVert .\rVert_\gamma$ on $H(\gamma_n)$ (see theorem \ref{thm_decomposition}, assertion (6)), we get 
	\begin{equation*}
	B_{H(\gamma)}\cap span(h_0,...,h_{n-1})^\perp=B_{H(\gamma_n)}.
	\end{equation*} But, for $h\in H(\gamma_n)$, $\langle h,f\rangle_{X,X^*}=\langle h,R_{\gamma_n}f\rangle_{\gamma_n}$ and $\sup_{h\in B_{H(\gamma_n)}}\langle h,f\rangle_{X,X^*}$ is attained for $h=\frac{R_{\gamma_n}f}{\lVert R_{\gamma_n}f\rVert_{\gamma_n}}$ (if $R_{\gamma_n}f\not=0$). Thus, $\sup_{h\in B_{H(\gamma_n)}}\langle h,f\rangle_{X,X^*}=\sqrt{\langle R_{\gamma_n}f,f\rangle_{X,X^*}}$.
\end{proof}

\begin{lemma}
	The sequence $(\lambda_n)_{n\geq 0}$ is non-increasing and $\lambda_n\to 0$.
	\label{lem_decroissance}
\end{lemma}

\begin{proof}[Proof of lemma \ref{lem_decroissance}]
	By lemma \ref{lem_doublesup} and the expression \ref{equation_doublesup}, we see that $\lambda_{n+1}\leq \lambda_n$. Moreover, $(h_n)$ is an orthonormal system in $H(\gamma)$, hence
	\begin{equation*}
	\forall f\in X^*,\;\langle h_n,f\rangle_{X,X^*}=\langle h_n,R_{\gamma}f\rangle_\gamma\to 0,
	\end{equation*}
	as a consequence of Bessel's inequality. In other words, we have that $h_n\rightharpoonup 0$ for the weak topology of $X$. Since the unit ball of $H(\gamma)$ is precompact in $X$ (corollary 3.2.4 p.101 in \cite{Bog1998}), we can extract a subsequence $(h_{n_k})_k$ such that $h_{n_k}\to_k h_\infty$ for the strong topology of $X$. By unicity of limit in the topological vector space $X$ equipped with the weak topology, we deduce that $h_\infty=0$ in $X$. Therefore, $\lVert h_{n_k}\rVert_X=\sqrt{\lambda_{n_k}}\to_k 0$, which ends the proof. 
\end{proof}

The two above lemmas are the ingredients to prove now that the orthonormal family $(h_n)_n$ is a Hilbert basis of $H(\gamma)$ in $R_{\gamma}(X^*)$ as it is discussed in \cite{Vak1993}.

\begin{theorem}
	$(h_n)_{n \geq 0}=(R_\gamma h_n^*)_{n \geq 0}$ is a Hilbert basis of $H(\gamma)$.
	\label{thm_HilbertBasis}
\end{theorem}

\begin{proof}[Proof of theorem \ref{thm_HilbertBasis}]
	Let $h\in H(\gamma)$ such that $\forall n\in\mathbb{N},\;\langle h,h_n\rangle_\gamma=0$. Then, using lemma \ref{lem_doublesup}, we have
	\begin{equation*}
	\forall n\in\mathbb{N},\;\forall f\in B_{X^*},\;\langle h,f\rangle_{X,X^*}\leq \sqrt{\lambda_n}\lVert h\rVert_\gamma,
	\end{equation*}
	which implies that $\langle h,f\rangle_{X,X^*}=0$ for all $f \in X^*$. Therefore, $h = 0$ and $span(h_n,\; n \geq 0)$ is dense in $H(\gamma)$.
\end{proof}

We give now the two claimed results of this paper.

\begin{corollary}
	The covariance operator can be decomposed as follows
	\begin{equation*}
	R_\gamma=\sum_{n \geq 0} \lambda_n\langle x_n,.\rangle_{X,X^*} x_n,
	\end{equation*}
	where the convergence is in $\mathcal{L}(X^*,X)$. More precisely, the nth step truncation error is
	\begin{equation*}
	\left\lVert R_\gamma-\sum_{k=0}^{n}\lambda_k\langle x_k,.\rangle_{X,X^*}x_k\right\rVert=\lambda_{n+1},
	\end{equation*}
	where $\left\lVert \cdot \right\rVert$ stands for the operator norm in $\mathcal{L}(X^*,X)$.
	\label{cor_estimation}
\end{corollary}

\begin{proof}[Proof of corollary \ref{cor_estimation}]
	From theorem \ref{thm_HilbertBasis}, we know that $(h_n)_n$ is a Hilbert basis of $H(\gamma)$. It suffices to write
	\begin{equation*}
	\forall f\in X^*,\;R_\gamma f=\sum_{n\geq 0}\langle R_\gamma f,h_n\rangle_\gamma h_n,
	\end{equation*}
	and use the reproducing property. The truncation error norm is
	\begin{equation*}
	\left\lVert R_\gamma-\sum_{k=0}^{n}\lambda_k\langle x_k,.\rangle_{X,X^*} x_k\right\rVert=\sup_{f\in B_{X^*}}\lVert R_{\gamma_{n+1}}f\rVert_X.
	\end{equation*}
	But,
	\begin{equation*}
		\lVert R_{\gamma_{n+1}}f\rVert_X=\sup_{g\in B_{X^*}}\langle R_{\gamma_{n+1}}f,g\rangle_{X,X^*}\leq \lambda_{n+1}
	\end{equation*}
	by the Cauchy-Schwarz inequality. Since $R_{\gamma_{n+1}}f_{n+1}=\lambda_{n+1}x_{n+1}$ and $\lVert x_{n+1}\rVert_X=1$, we have $\lVert R_{\gamma_{n+1}}f_{n+1}\rVert_X=\lambda_{n+1}$. Hence
	\begin{equation*}
		\left\lVert R_\gamma-\sum_{k=0}^{n}\lambda_k\langle x_k,.\rangle_{X,X^*} x_k\right\rVert=\lambda_{n+1}\to 0.
	\end{equation*}
\end{proof}

\begin{corollary} Remind the definition $h_{n}^{*} = \sqrt{\lambda_n}^{-1} x_{n}^{*}$ with $x_{n}^{*}=(I-P_{n-1})^*f_n$ for $n \geq 1$ and $x_{0}^{*}=f_0$. Then, we have the decomposition in $X$
	\begin{equation*}
	x=\sum_{n}\langle x,h_n^*\rangle_{X,X^*}h_n,\;\gamma\;a.e.,
	\end{equation*}	
	where the random variables $h_n^*$ are independent $\mathcal{N}(0,1)$. In equivalent form, let $(\xi_n)_n$ be a sequence of independent standard normal variables on $(\Omega,\mathcal{F},\mathbb{P})$. Then the random series
	\begin{equation*}
	\sum_{n}\sqrt\lambda_n\xi_n(\omega)x_n
	\end{equation*}
	defines a $X$-valued random Gaussian vector with distribution $\gamma$.
	\label{cor_KL}  
\end{corollary}

\section{Decomposition of the classical Wiener measure}

Let $\gamma$ be the standard Wiener measure on $X=\mathcal{C}([0,1],\mathbb{R})$, the space of all real continuous functions on the interval $[0,1]$ which is a Banach space if equipped with the supremum norm. The Riesz-Markov representation theorem allows to identify $X^*$ with the linear space of all bounded signed measures on $[0,1]$ equipped with the norm of total variation. In this context, the dual pairing is
\begin{equation*}
\forall x \in X,\;\forall\mu\in X^*,\;\langle x,\mu\rangle_{X,X^*}=\int_{0}^{1}x(t)\mu(dt).
\end{equation*}
The Cameron-Martin space associated to $\gamma$ is the usual Sobolev space $H_0^1([0,1],\mathbb{R})$, defined by
\begin{equation*}
H_0^1([0,1],\mathbb{R})=\left\lbrace f\in X,\;\forall t\in[0,1],\;f(t)=\int_{0}^{t}f'(s)ds,\;f'\in L^2([0,1],\mathbb{R})\right\rbrace
\end{equation*}
and associated inner product $\langle f_1,f_2\rangle_\gamma=\langle f_1',f_2'\rangle_{L^2}$. The covariance operator $R_\gamma$ satisfies
\begin{equation*}
\langle R_\gamma \mu,\mu\rangle_{X,X^*}=Var\left(\int_{0}^{1}W_t\mu(dt)\right)
\end{equation*}
where $(W_t)_{t\in[0,1]}$ is the standard Wiener process.  Using Fubini's theorem, we easily get
\begin{equation*}
\langle R_\gamma \mu,\mu\rangle_{X,X^*}=\iint_{[0,1]^2}t\wedge s\mu(dt)\mu(ds)= \int_{0}^{1}\mu([u,1])^2du.
\end{equation*}
Hence, $(R_\gamma\mu)'(t)=\mu([t,1])$ almost everywhere in $[0,1]$, and $R_\gamma\mu:t\in[0,1]\to\int_{0}^{t}\mu([u,1])du$. Consider now the initial step of the decomposition, that is find $f_0=\mu_0\in B_{X^*}$ such that
\begin{equation*}
\langle R_\gamma\mu_0,\mu_0\rangle_{X,X^*}=\sup_{\mu\in B_{X^*}}\langle R_\gamma\mu,\mu\rangle_{X,X^*}.
\end{equation*}
Since $\forall \mu\in B_{X^*}, \forall u\in[0,1],\;\vert\mu([u,1])\vert\leq 1$, the unique measure (up to sign) into $B_{X^*}$ maximizing $\langle R_\gamma\mu,\mu\rangle_{X,X^*}=\int_{0}^{1}\mu([u,1])^2du$ is $\mu_0=\delta_1$. Moreover,
\begin{equation*}
\lambda_0=\langle R_\gamma\mu_0,\mu_0\rangle_{X,X^*}=Var(W_1)=1
\end{equation*}
is the variance of the Wiener process at point $t=1$. Since $\mu\to\langle R_\gamma\mu,\mu\rangle_{X,X^*}$ is a non-negative quadratic functional, an usual argument shows directly that $\mu_0$ must be an extremal point of $B_{X^*}$. Thus $\mu_0=\delta_{t_0}$ for some point $t_0\in[0,1]$. And clearly, $t_0=1$, corresponding to the maximum of variance of the Wiener process. So, we have $\lambda_0=1$, $f_0=\mu_0=\delta_1$. Using the fact that 
\begin{equation*}
R_\gamma\delta_t:s\in[0,1]\to\langle R_\gamma\delta_t,\delta_s\rangle_{X,X^*}=Cov(W_t,W_s)=t\wedge s,
\end{equation*}
we get $x_0=(t\in[0,1] \to t)$ and $h_0=x_0$ (since $\lambda_0=1$). Now, we have $P_0x:t\in[0,1] \to\langle x,f_0\rangle_{X,X^*}x_0(t)=x(1)t$ and $(I-P_0)x$ is the function $t\in[0,1]\to x(t)-x(1)t$. From this, we see that $\gamma_1=\gamma\circ (I-P_0)^{-1}$ is the Gaussian measure associated to the Brownian bridge $(B_t)_{t\in[0,1]}$ with covariance kernel
\begin{equation*} 
K_1:(t,s)\in[0,1]^2\to Cov(B_t,B_s)=t\wedge s-ts.
\end{equation*}
Using now the fact that $\mu\to\langle R_{\gamma_1}\mu,\mu\rangle_{X,X^*}$ is a non-negative quadratic functional, we see that $f_1=\mu_1=\delta_{t_1}$ where $t_1=\frac{1}{2}$ is the maximum of variance of the Brownian bridge $B$. Hence, we get $\lambda_1=\frac{1}{4}$, $x_1=(t\to 4(t\wedge \frac{1}{2}-\frac{t}{2}))$ (by the relation $\lambda_1 x_1=R_{\gamma_1}\delta_{\frac{1}{2}}$) and $h_1=\frac{1}{2}x_1$. Furthermore, $x_1^* = \delta_{t_1} - \frac{1}{2}\delta_{t_0}$ and $\gamma_2=\gamma\circ (I-P_1)^{-1}$ is the Gaussian distribution of the process $(I - P_1)W : t \to W_t - W_1x_0(t) - (W_{\frac{1}{2}} - \frac{1}{2}W_1)x_1(t)$. By the assertion \ref{assertion_conditionnement} of theorem \ref{thm_decomposition}, $\gamma_2$ is the conditional distribution of $W$ given $W_1=0, W_\frac{1}{2}=0$. Using this interpretation, scale-invariance and spatial Markov properties of the Wiener process, we immediately get 
\begin{equation*} 
\lambda_n = \frac{1}{2^{p+2}}\;\text{ for }n = 2^p + k,\;k = 0, ..., 2^p-1 \;\text{ and }p \geq 0.
\end{equation*}
Furthermore, the Hilbert basis $(h_n)_{n\geq 0}$ of $H(\gamma)$ is given by $h_0(t)=t$ and $h_n(t)=\int_{0}^{t}h_n'(s)ds,\;n\geq 1$, where
\begin{equation*}
	h_n'(s)=\left\lbrace
	\begin{matrix}
	\sqrt{2^p}\text{ for }\frac{2k}{2^{p+1}}\leq s\leq \frac{2k+1}{2^{p+1}}\\
	-\sqrt{2^p}\text{ for }\frac{2k+1}{2^{p+1}}< s\leq \frac{2k+2}{2^{p+1}}\\
	0\text{ otherwise}
	\end{matrix}\right.,
\end{equation*}
if $n=2^{p}+k,\;k=0,...,2^p-1$ and $p\geq 0$. The family $(h_n')_{n\in\mathbb{N}}$ is the usual Haar basis of $L^2([0,1],\mathbb{R})$. The functions $(x_n)_{n\geq 0}$ are Schauder's functions
\begin{equation*}
	x_n(t)=\sqrt{2^{p+2}}h_n(t)
\end{equation*}
corresponding to \textit{hat functions} of height $1$ and lying above the intervals $\left[\frac{k}{2^{p}},\frac{k+1}{2^{p}}
\right]$ $(n=2^p+k)$. The resulting decomposition $\sum_{n}\sqrt{\lambda_n}\xi_n(\omega)x_n$ is the famous L\'evy-Ciesielski construction of Brownian motion on the interval $[0,1]$ (see \cite{McK1969}). The $8$ first steps (and the associated residual) of this decomposition are illustrated in figure \ref{Figure1}.

\begin{figure}
	\includegraphics[width=\textwidth]{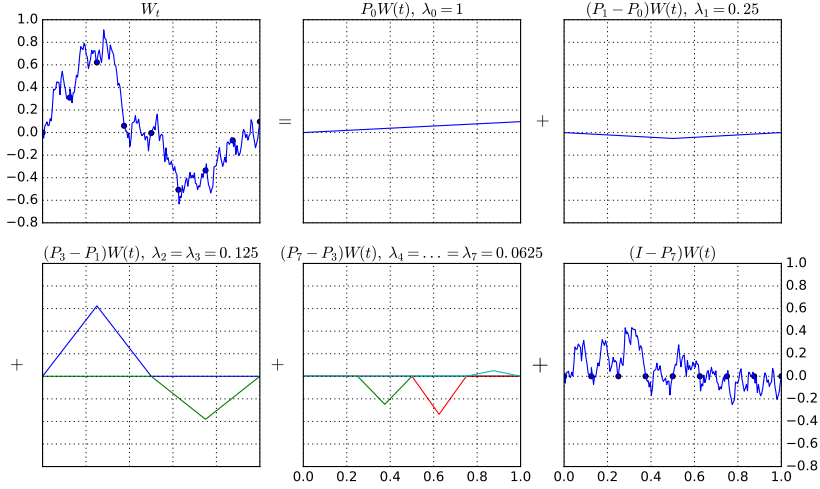}
	\caption{Decomposition of the standard Wiener measure on the 8 first steps.}
	\label{Figure1}
\end{figure}

\section{Comments}
	\begin{enumerate}
	\item For $\gamma$ a Gaussian measure on a separable Hilbert space $X$, corollary \ref{cor_estimation} is equivalent to the spectral theorem applied to the self-adjoint compact operator $R_\gamma$. In the Banach case, corollary \ref{cor_estimation} says that
	\begin{equation*}
		R_\gamma=\sum_{n \geq 0}\lambda_n\langle x_n,.\rangle_{X,X^*} x_n,
	\end{equation*}
	where $(\lambda_n)_n$ is a non-increasing sequence that converges to zero and $(x_n)_n$ is a sequence of unit norm vectors in $X$ and orthogonal in $H(\gamma)$. Furthermore, we have the same formula for the error (see comments below of its importance for applications):
	\begin{equation*}
	\left\lVert R_\gamma-\sum_{k=0}^{n}\lambda_k\langle x_k,.\rangle_{X,X^*}x_k\right\rVert=\lambda_{n+1}.
	\end{equation*} 
	Interpretation of the pairs $(\lambda_n,x_n)$ for each $n$ is the following: for $n = 0$, $x_0$ is a (unit) direction vector for a line in $X$ that has the largest variance possible ($=\lambda_0$) by a projection of norm one (namely, the projection $P_0$ in theorem \ref{thm_decomposition}). Remark that $P_0$ of norm one means $P_0$ orthogonal or self-adjoint in the Hilbert case. By considering the measure $\gamma_1 = \gamma\circ(I-P_0)^{-1}$, the vector $x_1$ is the direction vector for a line in the subspace $(I-P_0)X$ that has the largest variance possible and so on. In the Hilbert case, this decomposition process is known as (functional) principal component analysis.
	\item In this work, we assume the Radon measure $\gamma$ to be Gaussian. By a slight modification of the proof of lemma \ref{lem_existence}, the decomposition is valid if we assume only $\int_X\left\lVert x \right\rVert^2 \gamma(dx)< +\infty$ and results have to be interpreted in a mean-square sense (in particular, independence becomes non correlation and last parts of lemma \ref{lem_existence} and theorem \ref{thm_decomposition} on conditioning are valid only in the Gaussian case).
	\item The random series representation $\sum_{n\geq 0}\sqrt\lambda_n\xi_n(\omega)x_n$ in corollary \ref{cor_KL} is a generalization of the Karhunen-Lo\`eve expansion based on the corresponding decomposition of the covariance operator $R_\gamma=\sum_{n\geq 0}\lambda_n\langle x_n,.\rangle_{X,X^*} x_n$. 
	\item The decomposition of the classical Wiener measure shows that 
	\begin{equation*}
	\sum_{n} \lambda_n = 1 + \frac{1}{4} + 2\times\frac{1}{8} + 4\times\frac{1}{16} + ...= +\infty
	\end{equation*}
	due to the "multiplicity" of the values $\lambda_n$. In the Hilbert case, this sum is always finite and is the trace of the operator $R_\gamma$. Furthermore, this finite-trace property is characteristic of Gaussian measures on Hilbert spaces. Such a characterization in the Banach case is still an open problem.
	\item Gaussian hypothesis is motivated by applications both in Gaussian process regression (or Kriging, see \cite{Ras2006}) and Bayesian inverse problems (\cite{Stu2010}). As theorem \ref{thm_decomposition} indicates, we are interested in an efficient algorithm to construct a design of experiments (see \cite{Fed1972}) or a training set (functionals $(f_n)_n$ or, equivalently, $(x_n^{*})_n$). Error expression $\left\lVert R_\gamma-\sum_{k=0}^{n}\lambda_k\langle x_k,.\rangle_{X,X^*} x_k\right\rVert=\sup_{f\in B_{X^*}}\lVert R_{\gamma_{n+1}}f\rVert_X=\lambda_{n+1}$ in corollary \ref{cor_estimation} says that we have a precise quantification of uncertainty in terms of confidence interval in  the Gaussian case.
	\end{enumerate}

\section{Conclusion}

In this work, we suggest a Karhunen-Lo\`eve expansion for a Gaussian measure on a separable Banach space based on a corresponding decomposition of its covariance operator. In some sense, this decomposition generalizes the Hilbert case. L\'evy's construction of Brownian motion appears to be a particular case of such an expansion. Finally, we believe that this result will be useful both in pure and applied mathematics since it provides a canonical representation of Gaussian measures on separable Banach spaces.


\bibliographystyle{amsplain}
\bibliography{Bib}

\appendix

\section*{Proofs}

\begin{proof}[Proof of lemma \ref{lem_split}]
	\begin{enumerate}
		\item Since $R_\gamma f_0 = \lambda_0x_0$, we have
		\begin{equation*}
		\lambda_0\langle x_0,f_0\rangle_{X,X^*}=\langle R_\gamma f_0,f_0\rangle_{X,X^*}=\lambda_0
		\end{equation*}
		and $\lambda_0>0$ implies $\langle x_0,f_0\rangle_{X,X^*}=1$. The second equality is obtained from the definition of $h_0$ and the reproducing property:
		\begin{equation*}
		\lVert h_0\rVert_\gamma^2=\langle h_0,h_0\rangle_\gamma=\langle x_0,\lambda_0x_0\rangle_\gamma=\langle x_0,R_\gamma f_0\rangle_\gamma=\langle x_0,f_0\rangle_{X,X^*}=1.
		\end{equation*}
		\item Since $P_0x_0=\langle x_0,f_0\rangle_{X,X^*}x_0=x_0$, we have $P_0^2=P_0$ and $P_0$ is clearly the projection onto $\mathbb{R}x_0$ along the null space of $f_0\in X^*$. Now, if $h\in H(\gamma)$, we get by the reproducing property:
		\begin{equation*}
		P_0h=\langle h,R_\gamma f_0\rangle_\gamma x_0=\langle h,\lambda_0 x_0\rangle_\gamma x_0=\langle h,h_0\rangle_\gamma h_0=Q_0h.
		\end{equation*}
		\item As bounded linear transformations of a (centered) Gaussian measure, both $\gamma_{\lambda_0}$ and $\gamma_1$ are (centered) Gaussian measures. Consider the decomposition in $X^*$:
	\begin{equation*}
	f=P_0^*f+(I-P_0^*)f.
	\end{equation*}
Now, the random variable $P_0^*f=\langle x_0,f\rangle_{X,X^*}f_0$ is Gaussian with variance $\langle R_{\gamma_{\lambda_0}}f,f\rangle_{X,X^*}=\lambda_0\langle x_0,f\rangle_{X,X^*}^2$ and
$(I-P_0^*)f=f-\langle x_0,f\rangle_{X,X^*}f_0$ is Gaussian with variance $\langle R_{\gamma_1}f,f\rangle_{X,X^*}$. To show that $P_0^*f$ and $(I-P_0^*)f$ are independent, we compute their covariance: 
	\begin{equation*}
	\begin{split}
	&\int_X\langle x,P_0^*f\rangle_{X,X^*}\langle x,(I-P_0^*)f\rangle_{X,X^*}\gamma(dx)\\
	&=\int_{X}\langle x_0,f\rangle_{X,X^*}\langle x,f_0\rangle_{X,X^*}\left(\langle x,f\rangle-\langle x_0,f\rangle\langle x,f\rangle\right)\gamma(dx)\\
	&=\langle x_0,f\rangle_{X,X^*}\langle R_\gamma f_0,f\rangle_{X,X^*}-\lambda_0\langle x_0,f\rangle_{X,X^*}^2\\
	&=0.
	\end{split}
	\end{equation*}
Using the characteristic function of $\gamma$, we get by independence
	\begin{equation*}
	\hat{\gamma}(f)=\int_X e^{i\langle x,(P_0^*f+(I-P_0^*)f)\rangle_{X,X^*}}\gamma(dx)=\hat{\gamma_{\lambda_0}}(f)\hat{\gamma_1}(f).
	\end{equation*}
This proves $\gamma = \gamma_{\lambda_0}*\gamma_1 $ and also $R_\gamma=R_{\gamma_{\lambda_0}}+R_{\gamma_1}$.
	\item Consider the orthogonal decomposition $H(\gamma)=\mathbb{R}h_0\oplus H_1$ where $H_1=(\mathbb{R}h_0)^\perp$. Since $R_{\gamma_{\lambda_0}}f=\lambda_0\langle x_0,f\rangle_{X,X^*}x_0=\langle R_\gamma f,h_0\rangle_\gamma h_0$ is the orthogonal projection of $R_\gamma f$ onto $\mathbb{R}h_0$, we see that $R_\gamma f= R_{\lambda_0}f+R_{\gamma_1}f$ is the corresponding orthogonal decomposition of $R_\gamma f$. Therefore, by the Pythagorean theorem,
	\begin{equation*}
	\lVert R_\gamma f\rVert_\gamma^2=\lVert R_{\gamma_0}f\rVert_\gamma^2+\lVert R_{\gamma_1}f\rVert_\gamma^2.
	\end{equation*}
Now, using the relation $R_{\gamma_{\lambda_0}}f=\lambda_0\langle x_0,f\rangle_{X,X^*}x_0$, we get  $\lVert R_{\gamma_{\lambda_0}}f\rVert_\gamma^2=\lambda_0\langle x_0,f\rangle_{X,X^*}^2=\langle R_{\gamma_{\lambda_0}}f,f\rangle_{X,X^*}$ (= $\lVert R_{\gamma_{\lambda_0}}f\rVert_{\gamma_{\lambda_0}}^2$), thus 
	\begin{equation*}
	\lVert R_{\gamma_1}f\rVert_\gamma^2=\langle R_\gamma f,f\rangle_{X,X^*}-\langle R_{\lambda_0}f,f\rangle_{X,X^*}=\langle R_{\gamma_1}f,f\rangle_{X,X^*}.
	\end{equation*}
	Using the reproducing property in the Cameron-Martin space $H(\gamma_1)$, we get  $\lVert R_{\gamma_1}f\rVert_\gamma^2=\lVert R_{\gamma_1}f\rVert_{\gamma_1}^2$.  Since       $R_{\gamma_1}(X^*)$ is dense in $H(\gamma_1)$, we conclude that $H(\gamma_1)$ is a subspace of $H_1 $ and, in particular, $\langle .,.\rangle_{\gamma_1}=\langle .,.\rangle_\gamma$. Finally, $H(\gamma_1) = H_1 $ by density of $R_\gamma(X^*)$ in $H(\gamma)$.
	\item Using $\gamma =\gamma_{\lambda_0}*\gamma_1$, we can write for all $B\in\mathcal{B}(X)$:
	\begin{equation*}
	\gamma(B)=\int_X\gamma_1(B-tx_0)\frac{e^{-\frac{t^2}{\lambda_0}}}{\sqrt{2\pi\lambda_0}}dt.
	\end{equation*}
	Since $f_0\sim N(0,\lambda_0)$, we deduce that $\gamma(B\vert f_0=t)=\gamma_1(B-tx_0)$ (as a regular conditional probability).
\end{enumerate}
\end{proof}

\begin{proof}[Proof of theorem \ref{thm_decomposition}]
	\begin{enumerate}
		\item For $n\in\mathbb{N}$, $\lVert h_n\rVert_{\gamma}=1$ by construction. If $n<m$, remark that $h_n\in span(h_0, ..., h_{m-1}) = H(\gamma_{m})^\perp$ to get 
		 $\langle h_n,h_m\rangle_\gamma=0$.        
		\item By definition of $x_n$, we have $\langle x_n,f_n\rangle_{X,X^*}=1$. Now, the reproducing property gives
		\begin{equation*}
		\forall f\in B_{X^*},\;\langle x_n,f\rangle_{X,X^*}=\langle x_n,R_{\gamma_n} f\rangle_{\gamma_n}\leq \lVert x_n\rVert_{\gamma_n}\sqrt{\langle R_{\gamma_n}f,f\rangle_{X,X^*}}.
		\end{equation*}
Using the relations $\langle R_{\gamma_n}f,f\rangle_{X,X^*}\leq\lambda_n$ and $\lVert \sqrt{\lambda_n} x_n\rVert_{\gamma_n}=\lVert h_n\rVert_{\gamma}=1$, we get $\langle x_n,f\rangle_{X,X^*}\leq1$. This proves that $\lVert x_n\rVert_{X^*}=\langle x_n,f_n\rangle_{X,X^*}=1$. \\
For $k > l$, $h_k\in H(\gamma_{l})$ and the reproducing property gives
		\begin{equation*}
		\sqrt{\lambda_{k}}\langle x_k,f_l\rangle_{X,X^*}=\langle h_k,R_{\gamma_l}f_l\rangle_{\gamma_l}=\sqrt{\lambda_l}\langle h_k,h_l\rangle_{\gamma}=0.
		\end{equation*}
Hence $\langle x_k,f_l\rangle_{X,X^*} = 0$ since $\lambda_{k} > 0$.
		\item For $h\in H(\gamma)$, we have
		\begin{equation*}
		Q_nh=\sum_{k=0}^{n}\langle h,\lambda_{k}x_k\rangle_{\gamma}x_k=\sum_{k=0}^{n}\langle h,R_{\gamma_k}f_k\rangle_{\gamma}x_k.
		\end{equation*}
		According to the orthogonal decomposition
		\begin{equation*}
			H(\gamma)=span(h_0,...,h_{k-1})\oplus H(\gamma_{k}),
		\end{equation*}
		we get that
		\begin{equation*}
			\langle h,R_{\gamma_k}f_k\rangle_\gamma=\langle h-Q_{k-1}h,R_{\gamma_k}f_k\rangle_{\gamma_k}=\langle h-Q_{k-1}h,f_k\rangle_{X,X^*},
		\end{equation*}
		which proves the result.
		\item Let $x\in X$ then $P_nx\in span(x_0,...,x_n)=range(Q_n)$ thus $P_nx\in H(\gamma)$ and $P_n(P_nx)=Q_n(P_n x)=P_n x$. Clearly, we have:
		\begin{equation*}
		\bigcap_{k=0}^n\ker(f_k)\subset\ker(P_n).
		\end{equation*}
		Conversely, if $P_nx=0$ then $P_kx=0$ for all $k\in[0,n]$ and $0=\langle x-P_{k-1}x, f_k\rangle=\langle x,f_k\rangle_{X,X^*}$, hence $\ker(P_n)\subset\bigcap_{k=0}^n\ker(f_k)$.
		\item Since $Q_n = P_n$ on $H(\gamma)$, remark first that 
		$R_{\gamma_{\lambda_0, ...,\lambda_n}}  = R_\gamma P_n^* = Q_n R_\gamma$ and also $R_{\gamma_{n+1}}  = R_\gamma (I -P_n)^* = (I - Q_n) R_\gamma$. In particular, $R_{\gamma_{\lambda_0, ...,\lambda_n}}f=\sum_{k=0}^n{\langle R_\gamma f,h_k\rangle_\gamma h_k}=\sum_{k=0}^n\lambda_k\langle x_k,f\rangle_{X,X^*}x_k$. Consider now the decomposition for $f \in X^*$:
		\begin{equation*}
	f=P_n^*f+(I-P_n)^*f. 
		\end{equation*}
		The random variable $P_n^*f$ is Gaussian with variance $\langle R_{\gamma_{\lambda_0, ...,\lambda_n}}f,f\rangle_{X,X^*}$ and $(I-P_n)^*f$ is Gaussian with variance $\langle R_{\gamma_{n+1}}f,f\rangle_{X,X^*}$. Since $\langle R_\gamma P_n^*f,(I-P_n)^*f\rangle_{X,X^*} = \langle (I-Q_n)Q_nR_\gamma f,f\rangle_{X,X^*}=0$, the random variables $P_n^*f$ and $(I-P_n)^*f$ are independent and we conclude as in lemma (3.2).
		\item The proof is similar to the proof of (4) in lemma (3.2). Introduce the space $H_{n+1}=span(h_0, ..., h_n)^\perp$, we have that $(R_{\gamma_{n+1}}(X^*),\langle .,. \rangle_{\gamma_{n+1}})$ is a subspace of $H_{n+1}$, which is sufficient to prove $H(\gamma_{n+1}) = H_{n+1}$ as Hilbert spaces.
		\item For $n \geq 0$ and $h \in H(\gamma)$, we write $\langle h,R_\gamma x_{n}^{*}\rangle_\gamma = \langle h, (I - P_{n-1})^{*}f_n \rangle_{X,X^*}$, thus
		$\langle h,R_\gamma x_{n}^{*}\rangle_\gamma =\langle (I - Q_{n-1})h, f_n \rangle_{X,X^*}=\langle (I - Q_{n-1})h, R_{\gamma_{n}}f_n \rangle_{\gamma}$. \\
		Using now the relation $R_{\gamma_{n}}f_n = \lambda_n x_n $, we finally get $\langle h,R_\gamma x_{n}^{*}\rangle_\gamma = \langle h,\lambda_n x_n\rangle_\gamma$, which proves $R_\gamma x_{n}^{*}=\lambda_n x_n$. In particular, $\langle R_\gamma x_{n}^{*}, x_{n}^{*} \rangle_{X,X^*} = \lambda_n$. In the same way, we get $\langle R_\gamma x_{m}^{*}, x_{n}^{*} \rangle_{X,X^*} = 0$ if $m \neq n$. Hence, the random variables $x_n^*$ are independent with respective variance $\lambda_n$. The computation of this sequence comes from the identity $P_n^{*}f=\sum_{k=0}^{n}\langle x_{k},f \rangle_{X,X^*}x_{k}^{*}$.
		\item This is a reformulation of the previous statement about the sequence $(x_{n}^{*})_n$.
		\item  This last assertion is a direct consequence of (5) and (7). 		
		
	\end{enumerate}
\end{proof}

\end{document}